\documentclass[11pt,reqno]{amsart}

\usepackage{amssymb,eucal}

\newtheorem{theorem}{Theorem}[section]

\newtheorem{corollary}[theorem]{Corollary}
\newtheorem{proposition}[theorem]{Proposition}

\numberwithin{equation}{section}

\theoremstyle{definition}

\newtheorem*{example*}{Example}
\newtheorem{example}[theorem]{Example}

\newtheorem{remark}[theorem]{Remark}
\newtheorem*{remark*}{Remark}

\newcommand\frsl{\mathfrak{sl}}
\newcommand\frgl{\mathfrak{gl}}
\newcommand\frso{\mathfrak{so}}

\newcommand{\calL}{\mathcal{L}}

\newcommand{\calT}{\mathcal{T}}
\newcommand{\calK}{\mathcal{K}}
\newcommand{\frg}{\mathfrak{g}}

\newcommand{\frd}{\mathfrak{d}}

\newcommand{\bF}{\mathbb{F}}

\newcommand{\bZ}{\mathbb{Z}}

\newcommand{\bH}{\mathbb{H}}

\newcommand{\subo}{_{\bar 0}}
\newcommand{\subuno}{_{\bar 1}}

\DeclareMathOperator{\Der}{\mathrm{Der}}
\DeclareMathOperator{\Inder}{\mathrm{Inder}}
\DeclareMathOperator{\Aut}{\mathrm{Aut}}
\DeclareMathOperator{\End}{\mathrm{End}}
\DeclareMathOperator{\Mat}{\mathrm{Mat}} \DeclareMathOperator{\trace}{trace}

\def\hregleta{\hrule height .5pt}
\def\hreglon{\hrule height1pt}
\def\vreglon{\vrule height 12pt width1pt depth 4pt}
\def\vregleta{\vrule width .5pt}

\begin{document}

\title{Derivations of the Cheng-Kac Jordan superalgebras}


\author[Elisabete Barreiro]{Elisabete Barreiro$^{\star}$}
\thanks{$^{\star}$ Supported by the Center of Mathematics of the
University of Coimbra/FCT}
\address{CMUC, Departamento de Matem\'atica, Universidade de Coimbra,
3001-454 Coimbra, Portugal} \email{mefb@mat.uc.pt}

\author[Alberto Elduque]{Alberto Elduque$^{\star\star}$}
\thanks{$^{\star\star}$ Supported by the Spanish
Ministerio de Educaci\'{o}n y Ciencia and FEDER (MTM 2007-67884-C04-02) and by
the Diputaci\'on General de Arag\'on (Grupo de Investigaci\'on de \'Algebra)}
\address{Departamento de Matem\'aticas e
Instituto Universitario de Matem\'aticas y Aplicaciones, Universidad de
Zaragoza, 50009 Zaragoza, Spain} \email{elduque@unizar.es}

\author[Consuelo Mart\'{\i}nez]{Consuelo Mart\'{\i}nez$^{\star\star\star}$}
\thanks{$^{\star\star\star}$ Supported by the Spanish
Ministerio de Educaci\'{o}n y Ciencia and FEDER (MTM 2007-67884-C04-01) and the Principado de Asturias (IB-08-147)}
\address{Departamento de Matem\'aticas, Universidad de Oviedo, 33007 Oviedo,
Spain}
\email{chelo@orion.ciencias.uniovi.es}


\begin{abstract}
The derivations of the Cheng-Kac Jordan superalgebras are studied. It is shown
that, assuming $-1$ is a square in the ground field, the Lie superalgebra of
derivations of a Cheng-Kac Jordan superalgebra is  isomorphic to the Lie
superalgebra obtained from a simpler Jordan superalgebra (a Kantor double
superalgebra of vector type) by means of the Tits-Kantor-Koecher construction.
This is done by exploiting the $S_4$-symmetry of the Cheng-Kac Jordan
superalgebra.
\end{abstract}

\maketitle


\section{Introduction}

Cheng-Kac superalgebras appeared for the first time in \cite{CK97}, where Cheng
and Kac introduced a new superconformal algebra that they denoted as $CK_6$ in
connection with the classification of the simple conformal Lie superalgebras of
finite type.

This superalgebra was considered independently in \cite{GLS01}.

The general construction, in a Jordan context, appears in the classification of
simple finite dimensional unital Jordan superalgebras with a non-semisimple
even part \cite{MZ01}. The construction involves an arbitrary associative
commutative superalgebra $Z$,  and an even derivation $\delta:Z \rightarrow Z$.
In order to get simple superalgebras, $Z$ must be $\delta$-simple, that is, it
does not contain any proper $\delta$-invariant ideal.  We will denote by
$J=JCK(Z,\delta)$  the Jordan superalgebra obtained in this way, while
$CK(Z,\delta)$ will denote the corresponding centerless Lie superalgebra
obtained via the Tits-Kantor-Koecher process.

If we consider $Z = F[t,t^{-1}]$ the algebra of Laurent polynomials and $\delta
= \frac{d\ }{dt}$ the usual derivative, then $CK(Z,\delta)$ gives the
superconformal algebra $CK_6$ found by Cheng and Kac.

The Jordan Cheng-Kac superalgebra is a free module of dimension 8 over $Z$, and
hence, if the dimension of $Z$ over the ground field $\bF$ is finite, the
dimension of $JCK(Z,\delta)$ over the ground field is $8\dim_{\bF}Z$. However,
the dimension of the Lie superalgebra $CK(Z,\delta)$ is not so obvious.

In \cite{Bar10}, the first author studies the dimension of the  Cheng-Kac Lie
superalgebra in the case that $Z$ is the algebra of truncated polynomials over
the prime field in characteristic $3$ and $\delta$ is the usual derivation.

But her results on the dimension of the Cheng-Kac superalgebra extend to the
general case. The Lie superalgebra of derivations of the Jordan superalgebra
$JCK(Z,\delta)$ will be described explicitly here, as well as its subalgebra of
inner derivations. In doing so, the Kantor double superalgebra of vector type
$Z\oplus Zx$ and its derivation algebra will play an essential role.

Moreover, these results will be used to prove that the Lie superalgebra of
derivations of $JCK(Z,\delta)$, or its subalgebra of inner derivations, can be
obtained through Tits version of the Tits-Kantor-Koecher construction applied
to the subalgebra $Z\oplus Zx$ of $JCK(Z,\delta)$.  This will be done by
exploiting a symmetry over the symmetric group of degree $4$ of the Cheng-Kac
Jordan superalgebra.


\section{Preliminaries}

Let $Z$ be a unital commutative and associative algebra over a field $\bF$ of
characteristic $\ne 2$. Let $\delta\in\Der Z$ be a derivation such that
$Z\delta(Z)=Z$. In particular, this is the situation when $\delta\ne 0$ and $Z$
is $\delta$-simple.

The \emph{Cheng-Kac Jordan superalgebra} $J=JCK(Z,\delta)$ is the Jordan
superalgebra $J=J\subo\oplus J\subuno$ where both $J\subo$ and $J\subuno$ are
free $Z$-modules of rank $4$: $J\subo=Z1\oplus(\oplus_{i=1}^3 Zw_i)$,
$J\subuno=Zx\oplus(\oplus_{i=1}^3Zx_i)$, $J\subo$ is a unital Jordan algebra
over $Z$, and it is isomorphic to the scalar extension $\bigl(\bF
1\oplus(\oplus_{i=1}^3\bF w_i)\bigr)\otimes_\bF Z$ of the Jordan algebra over
$\bF$ of a quadratic form with $w_1^2=w_2^2=1=-w_3^2$ and $w_iw_j=0$ for $i\ne
j$. The multiplication of even and odd elements and of odd elements is given by
the following tables:

\[
\vbox{\offinterlineskip \halign{\hfil$#$\enspace\hfil&#\vreglon
 &\hfil\enspace$#$\enspace\hfil&#\vregleta
 &\hfil\enspace$#$\enspace\hfil\cr
 J\subo J\subuno&&gx&&gx_j\cr
 \noalign{\hreglon}
 f&&(fg)x&&(fg)x_j\cr
 \noalign{\hregleta}
 fw_i&&\bigl(\delta(f)g\bigr)x_i&&-(fg)x_{i\times j}\cr
 }}
\qquad\qquad \vbox{\offinterlineskip \halign{\hfil$#$\enspace\hfil&#\vreglon
 &\hfil\enspace$#$\enspace\hfil&#\vregleta
 &\hfil\enspace$#$\enspace\hfil\cr
 J\subuno J\subuno&&gx&&gx_j\cr
 \noalign{\hreglon}
 fx&&\delta(f)g-f\delta(g)&&-(fg)w_j\cr
 \noalign{\hregleta}
 fx_i&&(fg)w_i&&0\cr
 }}
\]
\smallskip

\noindent where $x_{1\times 2}=-x_{2\times 1}=x_3$, $x_{1\times 3}=-x_{3\times
1}=x_2$, $x_{3\times 2}=-x_{2\times 3}=x_1$ and $x_{i\times i}=0$ for any
$i=1,2,3$.

Let us list some easy properties of $J$:

\begin{proposition}\label{pr:easyproperties}
Under the conditions above:
\begin{enumerate}
\item $Z$\,($=Z1$) is the center $Z(J\subo)$ of $J\subo$ (that is, the
    elements which commute and associate with any other elements, the
    commutativity being trivial here).

\item $J\subuno^2=J\subo$, as $(Zx)x=\delta(Z)$ is in $J\subuno^2$, and
    $(fx)(gx)=\delta(f)g-f\delta(g)=\delta(fg)-2f\delta(g)$, so that
    $f\delta(g)$ is in $J\subuno^2$, and hence $Z=Z\delta(Z)$ is contained
    in $J\subuno^2$. (Obviously $Zw_i$ is contained in $J\subuno^2$ for any
    $i=1,2,3$.)\newline
    Therefore any derivation (in the super sense) of $J$ is determined by
    its action on $J\subuno$, and hence by its action on $Z$, $x$, $x_1$,
    $x_2$ and $x_3$.

\item $J$ is $\bZ_2^2$-graded with
\begin{equation}\label{eq:Z22grading}
\begin{split}
J^{[\bar 0,\bar 0]}&=Z\oplus Zx,\\ J^{[\bar 1,\bar 0]}&=Zw_1\oplus Zx_1,\\
J^{[\bar 0,\bar 1]}&=Zw_2\oplus Zx_2,\\ J^{[\bar 1,\bar 1]}&=Zw_3\oplus
Zx_3.
\end{split}
\end{equation}

\item The subspace $Zx$ consists precisely of the elements $z\in J$ such
    that $zw_i=0$ for any $i=1,2,3$.
\end{enumerate}
\end{proposition}

\bigskip

Our aim is the determination of the Lie superalgebra of derivations of the
Cheng-Kac Jordan superalgebra.

Also, assuming $-1\in\bF^2$, it will be proved that this Lie superalgebra is
isomorphic to the  Tits-Kantor-Koecher Lie superalgebra attached to the
subalgebra $Z\oplus Zx$ of $J$. This will be done by exploiting a symmetry over
the symmetric group of degree $4$ of the Cheng-Kac Jordan superalgebra.

\bigskip


\section{The superalgebra of derivations $\Der(Z\oplus Zx)$}

In order to compute the Lie superalgebra of derivations of a Cheng-Kac Jordan
superalgebra $J=JCK(Z,\delta)$, we will start by computing the derivations of
the Kantor double superalgebra of vector type $K=Z\oplus Zx$, which is a
subalgebra of $JCK(Z,\delta)$.

Given any derivation $\mu$ of the algebra $Z$ such that $[\mu,\delta]\in
Z\delta$, let $a\in Z$ be such that $[\mu,\delta]=2a\delta$. Then there is a
unique even derivation $\check\mu$ of $Z\oplus Zx$ such that
$\check\mu\vert_Z=\mu$ and $\check\mu(x)=ax$.

Also, for any $a\in Z$, there is a unique odd derivation $\eta_a$ such that
$\eta_a(Z)=0$, $\eta_a(x)=a$. Moreover, if the characteristic of $\bF$ is $3$,
then for any $\mu\in Z\delta$, there is a unique odd derivation $\mu^-$ such
that $\mu^-(z)=\mu(z)x$ for any $z\in Z$, and $\mu^-(x)=0$.

Recall that the Lie superalgebra $\Inder(J)$ of inner derivations of a Jordan
superalgebra $J$ is the linear span of the derivations $D(a,b): c\mapsto
a(bc)-(-1)^{\bar a\bar b}b(ac)$.  This is an ideal of the Lie superalgebra
$\Der(J)$ of derivations.

\begin{proposition}\label{pr:DerZZx}
Let $K$ be the Jordan superalgebra $Z\oplus Zx$ (a subalgebra of the Cheng-Kac
Jordan superalgebra $JCK(Z,\delta)$). Then:
\begin{itemize}
\item The assignment $\mu\mapsto \check\mu$ gives an isomorphism
    $\{\mu\in\Der(Z): [\mu,\delta]\in Z\delta\}\rightarrow \Der(K)\subo$.
    Under this isomorphism $\Inder(K)\subo$ corresponds to $Z\delta$ (which
    is linearly isomorphic to $Z$).
\item If the characteristic of $\bF$ is $\ne 3$, then
    $\Der(K)\subuno=\Inder(K)\subuno=\{\eta_a: a\in Z\}$, which is linearly
    isomorphic to $Z$. \item If the characteristic of $\bF$ is $3$, then
    $\Inder(K)\subuno=\{\eta_a: a\in Z\}$ and
    $\Der(K)\subuno=\Inder(K)\subuno\oplus \{\mu^-: \mu\in Z\delta\}$.
\end{itemize}
\end{proposition}
\begin{proof}
Let $\gamma$ be an even derivation of $K$, then $\mu=\gamma\vert_Z$ is a
derivation of $Z$, and $\gamma(x)\in Zx$, so there is an element $a\in Z$ such
that $\gamma(x)=ax$. For $f,g\in Z$ we get:
\[
\begin{split}
&\gamma(fx)=\mu(f)x+f(ax)=(\mu(f)+af)x,\\[6pt]
&\gamma((fx)(gx))=\gamma(\delta(f)g-f\delta(g))=\mu(\delta(f)g-f\delta(g)),\\[4pt]
&\gamma(fx)(gx)+(fx)\gamma(gx)=((\mu(f)+af)x)(gx)+(fx)(\mu(g)+ag)x)\\
&\qquad\qquad\qquad
=\bigl(\delta\mu(f)+\delta(fa)\bigr)g-(\mu(f)+fa)\delta(g)\\
&\qquad\qquad\qquad\qquad
+\delta(f)(\mu(g)+ga)-f\bigl(\delta\mu(g)+\delta(ga)\bigr).
\end{split}
\]
With $g=1$ we obtain
$\mu\delta(f)=\delta\mu(f)+\delta(fa)+\delta(f)a-f\delta(a)=\delta\mu(f)+2a\delta(f)$,
that is, $[\mu,\delta]=2a\delta$. Hence $\gamma$ is determined by $\mu$:
$\gamma=\check\mu$, which is easily seen with the computations above to be
indeed an even derivation. As for the inner even derivations, note that
$D(Z,Z)=0$, while for any $f,g\in Z$,
\[
\begin{split}D(fx,gx):h\mapsto&(fx)(hgx)+(gx)(hfx)\ \text{($fx$ and $gx$ are
odd)}\\
 &=\delta(f)hg-f\delta(hg)+\delta(g)hf-g\delta(hf)=-2fg\delta(h).
 \end{split}
\]
Hence the restriction to $Z$ of $D(fx,gx)$ is $-2fg\delta$, and this gives the
isomorphism $\Inder(K)\subo\cong Z\delta$.

For any $f,g\in Z$, $D(f,gx)(Z)=0$ and
$D(f,gx)(x)=f\delta(g)-\delta(g)f+g\delta(f)=g\delta(f)$. That is,
\begin{equation}\label{eq:DfgxK}
D(f,gx):Z\mapsto 0,\quad x\mapsto g\delta(f).
\end{equation}
As $Z=Z\delta(Z)$, this shows that $\Inder(K)\subuno=\{\eta_a: a\in Z\}$.

Now, given any $\eta\in\Der(K)\subuno$, let $a=\eta(x)$ and change $\eta$ by
$\eta-\eta_a$. Hence we may assume $\eta\in\Der(K)\subuno$ and $\eta(x)=0$.
There is a linear map $\mu\in\End_\bF(Z)$ such that $\eta(f)=\mu(f)x$, and
$\mu$ is a derivation of $Z$. Also, for any  $f\in Z$,
$\eta(fx)=\eta(f)x=(\mu(f)x)x=\delta\mu(f)$. For any $f,g\in Z$ we have
$\eta\bigl((fx)(gx)\bigr)=\eta(fx)(gx)-(fx)\eta(gx)$ and:
\[
\begin{split}
&\eta\bigl((fx)(gx)\bigr)=\mu\bigl(\delta(f)g-f\delta(g)\bigr)x,\\
&\eta(fx)(gx)-(fx)\eta(gx)=\bigl(\delta\mu(f)g-f\delta\mu(g)\bigr)x.
\end{split}
\]
With $g=1$ we obtain $[\mu,\delta]=0$, and hence
\begin{equation}\label{eq:deltamu}
\delta(f)\mu(g)=\mu(f)\delta(g)
\end{equation}
for any $f,g\in Z$. Also $\eta\bigl(f(gx)\bigr)=\eta(f)(gx)+f\eta(gx)$, and
\[
\begin{split}
&\eta\bigl(f(gx)\bigr)=\eta\bigl((fg)x\bigr)=\delta\mu(fg),\\[4pt]
&\eta(f)(gx)+f\eta(gx)=(\mu(f)x)(gx)+f\delta\mu(g)\\ &\qquad
=\delta\mu(f)g-\mu(f)\delta(g)+f\delta\mu(g)=\delta\mu(fg)-2\mu(f)\delta(g)-\delta(f)\mu(g),
\end{split}
\]
so that
\begin{equation}\label{eq:deltamubis}
2\mu(f)\delta(g)+\delta(f)\mu(g)=0,
\end{equation}
for any $f,g\in Z$.

If the characteristic of $\bF$ is $\ne 3$, \eqref{eq:deltamu} and
\eqref{eq:deltamubis} give $\mu(f)\delta(g)=0$ for any $f,g\in Z$, so that
$\mu(Z)\delta(Z)=0$, and since $Z=Z\delta(Z)$, $\mu(Z)Z=0$, so that $\mu=0$,
thus getting $\Der(K)\subuno=\Inder(K)\subuno$.

However, if the characteristic of $\bF$ is $3$, then \eqref{eq:deltamubis} is
equivalent to \eqref{eq:deltamu}. As $Z=Z\delta(Z)$, there are elements
$a_i,b_i\in Z$ with $1=\sum a_i\delta(b_i)$ and for any $f\in Z$, using
\eqref{eq:deltamu} we compute:
\[
\mu(f)=\sum \mu(f)a_i\delta(b_i)=\sum \delta(f)a_i\mu(b_i)=\bigl(\sum
a_i\mu(b_i)\bigr)\delta(f).
\]
Thus $\mu\in Z\delta$. Conversely, if $\mu\in Z\delta$, then \eqref{eq:deltamu}
holds, and $\eta=\mu^-$ is indeed an odd derivation.
\end{proof}

\smallskip

We will denote by $\overline{\Der}(K)$ the subalgebra $\Der(K)\subo\oplus
\{\eta_a:a\in Z\}=\Der(K)\subo\oplus\Inder(K)\subuno$ of $\Der(K)$. If the
characteristic is $\ne 3$, then $\overline{\Der}(K)=\Der(K)$.

In spite of the different behavior in the previous proposition depending on the
characteristic being $3$ or $\ne 3$, the Lie superalgebra which plays a key
role in studying the Lie superalgebra of derivations of the Cheng-Kac Jordan
superalgebras is the Lie superalgebra $\overline{\Der}(K)$, whose definition
and behavior is independent of the characteristic.

\smallskip

\begin{example}\label{ex:FttpK}
Let $\bF$ be a field of prime characteristic $p>2$, let $Z=\bF[t:t^p=0]$ be the
algebra of truncated polynomials, and let $\delta=\frac{d\ }{dt}$ be the usual
derivative. Then $Z$ is $\delta$-simple. Let $K=Z\oplus Zx$ be the the Kantor
double superalgebra of vector type as above. Then $\Der(Z)=Z\delta$, because
any derivation is determined by its action on $t$. Hence
$\Der(K)\subo=\{\check\mu: \mu\in Z\delta\}=D(Zx,x)$, while
$\Inder(K)\subuno=D(Z,Zx)$. Since we have $Z=\delta(Z)\oplus\bF t^{p-1}$,
$D(Z,Zx)=\eta_Z=\eta_{\delta(Z)}\oplus\bF\eta_{t^{p-1}}=D(Z,x)\oplus \bF
D(t^{p-1},tx)$. If $p>3$, then $\Der(K)\subuno=\Inder(K)\subuno$, while for
$p=3$, $\Der(K)\subuno=\Inder(K)\subuno\oplus\{\mu^- :\mu\in \Der(Z)\}$.
\end{example}

\bigskip


\section{Derivations of the Cheng-Kac Jordan superalgebras}

\subsection{Even derivations} Let $J$ be the Cheng-Kac Jordan superalgebra
$J=JCK(Z,\delta)$ for a unital commutative and associative algebra $Z$ and a
derivation $\delta$ of $Z$ with $Z=Z\delta(Z)$. The computation of
$\Der(J)\subo$ will be done in several steps:

\medskip

\subsubsection{} The restriction
$\Der(J)\subo\rightarrow \Der(J\subo): \partial\mapsto \partial\vert_{J\subo}$
is one-to-one:

For if $\partial$ is an even derivation of $J$ with $\partial(J\subo)=0$, then
in particular $\partial(w_i)=0$, $i=1,2,3$, and hence $\partial(Zx)\subseteq
Zx$, as $Zx$ is the annihilator in $J$ of $\{w_1,w_2,w_3\}$. Therefore
$\partial\vert_K\in\Der(K)\subo$ with $\partial(Z)=0$. Proposition
\ref{pr:DerZZx} shows that $\partial\vert_K=0$, so $\partial(Zx)=0$ and
$\partial(Zx_i)=\partial\bigl((\delta(Z)Z)x_i\bigr)=\partial\bigl((Zw_i)(Zx)\bigr)=0$.

\medskip

\subsubsection{}
$\Der(J\subo)=D(J\subo,J\subo)\vert_{J\subo}\oplus\{\hat\partial: \partial\in
\Der(Z)\}$, where for any $\partial\in\Der(Z)$, $\hat\partial$ is the
derivation of $J\subo$ such that $\hat\partial(w_i)=0$, $i=1,2,3$, and
$\hat\partial\vert_Z=\partial$. (Note that $J\subo$ is the tensor product of
$Z$ and the Jordan algebra of a form over $\bF$: $\bF 1\oplus(\oplus_{i=1}^3\bF
w_i$).)

It is clear that both $D(J\subo,J\subo)\vert_{J\subo}$ and $\{\hat\partial:
\partial\in \Der(Z)\}$ are contained in $\Der(J\subo)$, and since
$D(J\subo,J\subo)(Z)=0$, the intersection of these two subspaces is trivial.

Now, let $\eta$ be a derivation of $J\subo$, then $\eta(Z)$ is contained in
$Z$, since $Z$ is the center of $J\subo$. Let $\partial=\eta\vert_Z$. Then
$\bar \eta=\eta-\hat\partial$ is a derivation of $J\subo$ which annihilates
$Z$, and hence $\bar\eta\in\Der_Z(J\subo)=\Der\bigl(\bF
1\oplus(\oplus_{i=1}^3\bF w_i)\bigr)\otimes_{\bF}Z$. But $\Der\bigl(\bF
1\oplus(\oplus_{i=1}^3\bF w_i)\bigr)=\Inder\bigl(\bF 1\oplus(\oplus_{i=1}^3\bF
w_i)\bigr)$ and the result follows.

\medskip

\subsubsection{}
Given any $\partial\in\Der(K)\subo$, Proposition \ref{pr:DerZZx} shows that the
restriction $\partial\vert_Z=\mu\in\Der(Z)$ satisfies $[\mu,\delta]=2a\delta$
for some $a\in Z$, such that $\partial(x)=ax$. Define $\tilde\partial$ by means
of
\begin{equation}\label{eq:tildepartial}
\tilde\partial\vert_{K}=\partial,\quad\tilde\partial(fw_i)=\mu(f)w_i,\quad
\tilde\partial(fx_i)=(\mu(f)-af)x_i,
\end{equation}
for any $f\in Z$ and $i=1,2,3$.

\medskip

\subsubsection{}
$\Der(J)\subo=D(J\subo,J\subo)\oplus\{\tilde\partial:\partial\in\Der(K)\subo\}$.

Indeed, both $D(J\subo,J\subo)$ and
$\{\tilde\partial:\partial\in\Der(K)\subo\}$ are contained in $\Der(J)\subo$,
and its intersection is trivial since $D(J\subo,J\subo)$ annihilates $Z$. Now,
if $\eta$ is an even derivation of $J$ we know from the previous assertions
that $\eta\vert_{J\subo}=\eta_1+\eta_2$ with $\eta_1\in
D(J\subo,J\subo)\vert_{J\subo}$ and $\eta_2=\hat\mu$ for some $\mu\in\Der(Z)$.
Since $\eta_1$ is the restriction to $J\subo$ of a derivation in
$D(J\subo,J\subo)$ we may subtract this derivation from $\eta$ and hence assume
that $\eta\vert_{J\subo}=\hat\mu$ for some $\mu\in\Der(Z)$. In particular we
have $\eta(w_i)=0$, $i=1,2,3$, and hence $\eta(Zx)\subseteq Zx$ since $Zx$ is
the annihilator of $\{w_1,w_2,w_3\}$. Therefore we get $\partial=\eta\vert_K\in
Der(K)\subo$. We conclude that
$\eta\vert_{J\subo}=\tilde\partial\vert_{J\subo}$, and hence
$\eta=\tilde\partial$ since the restriction map is one-to-one.

\medskip

\subsubsection{} $\Inder(J)\subo=D(J\subo,J\subo)\oplus D(x,Zx)$ and
$D(x,Zx)=\{\tilde\partial:  \partial=\check\mu\ \text{for}$ $\text{some $\mu\in
Z\delta$}\}$.

For any $f,g\in Z$, $D(x,fx)(w_i)=0$ (recall that $Zx$ annihilates the
$w_i$'s), and $D(x,fx)(g)=-\delta(fg)+(\delta(f)g-f\delta(g))=-2f\delta(g)$.
Hence $D(x,fx)=-2\tilde\partial$, with $\partial=(f\delta)\check{}$.

Consider the $\bZ_2^2$-grading of $J$ in \eqref{eq:Z22grading} and the induced
grading on $\Der(J)$. From the previous assertion, and since
$\{\tilde\partial:\partial\in \Der(K)\subo\}$ is contained in $\Der(J)^{[\bar
0,\bar 0]}$ it is enough to deal with the $[\bar 0,\bar 0]$ component. Using
that
\begin{equation}\label{eq:cyclic}
(-1)^{\bar z_1\bar z_3}D(z_1,z_2z_3)+(-1)^{\bar z_2\bar
z_1}D(z_2,z_3z_1)+(-1)^{\bar z_3\bar z_2}D(z_3,z_1z_2)=0,
\end{equation}
for any homogeneous $z_1,z_2,z_3\in J$ we obtain
\[
\begin{split}
D(Zx,Zx)&\subseteq D(J\subo,J\subo)\oplus D(x,Zx),\\ D(Zx_j,Zx_j)&\subseteq
D(J\subo,J\subo)+D(x_j,Zx_j)=D(J\subo,J\subo),
\end{split}
\]
as $D(x_j,fx_j)$ is $0$ on $Z$, $x$ and $x_i$, $i=1,2,3$, for any $f\in Z$.

\medskip

\subsubsection{} The restriction of $D(Z,J\subo)$ and of $D(Zw_i,Zw_i)$
($i=1,2,3$) to $J\subo$ is $0$, so $D(Z,J\subo)=0$, since the restriction map
is one-to-one. Also \eqref{eq:cyclic} gives $D(Zw_i,Zw_j)=D(w_i,Zw_j)$.
Therefore $D(J\subo,J\subo)=\oplus_{i=1}^3D(w_i,Zw_{i+1})$ (indices modulo
$3$).

\medskip

The next result summarizes most of the work in this subsection:

\begin{proposition}\label{pr:DerJ0}
Let $J$ be the Cheng-Kac Jordan superalgebra $JCK(Z,\delta)$ for a unital
commutative and associative algebra $Z$ and a derivation $\delta$ of $Z$ with
$Z=Z\delta(Z)$. Relative to the $\bZ_2^2$-grading induced by
\eqref{eq:Z22grading} we have
\[
\begin{split}
\Der(J)\subo^{[\bar 0,\bar 0]}&=\{\tilde\partial: \partial\in
\Der(K)\subo\},\\[2pt] \Inder(J)\subo^{[\bar 0,\bar
0]}&=D(x,Zx)=\{\tilde\partial: \partial\in\Inder(K)\subo\}\\
&\phantom{=D(x,Zx)}=\{\tilde\partial:  \partial=\check\mu\ \text{for some
$\mu\in Z\delta$}\},
\\[2pt]
\Der(J)\subo^{[\bar 1,\bar 0]}&=D(w_2,Zw_3),\\ \Der(J)\subo^{[\bar 0,\bar
1]}&=D(w_3,Zw_1),\\ \Der(J)\subo^{[\bar 1,\bar 1]}&=D(w_1,Zw_2).
\end{split}
\]
\end{proposition}

\begin{corollary}\label{co:DerJ0}
Let $J$ be the Cheng-Kac Jordan superalgebra $JCK(Z,\delta)$ for a unital
commutative and associative algebra $Z$ and a derivation $\delta$ of $Z$ with
$Z=Z\delta(Z)$. Then $\dim_\bF\Inder(J)\subo=4\dim_\bF Z=\dim_\bF J\subo$.
\end{corollary}
\begin{proof}
It is enough to take into account $\dim_\bF Z\delta=\dim_\bF Z$ (because
$f\delta=0$ if and only if $f\delta(Z)=0$ which gives $f=0$ as $Z=Z\delta(Z)$)
and also $\dim_\bF D(w_i,Zw_{i+1})=\dim_\bF Z$ for any $i=1,2,3$ (indices
modulo $3$), as $D(w_i,fw_{i+1})(w_i)=-(fw_{i+1})(w_i^2)=\pm fw_{i+1}$ for any
$f\in Z$, so that the linear maps $Z\rightarrow \Der(J)\subo: f\mapsto
D(w_i,fw_{i+1})$ are one-to-one.
\end{proof}

\begin{example}
Let $\bF$ be a field of prime characteristic $p>2$, let $Z=\bF[t:t^p=0]$ be the
algebra of truncated polynomials, and let $\delta=\frac{d\ }{dt}$ be the usual
derivative. Let $J=JCK(Z,\delta)$ be the associated Cheng-Kac Jordan
superalgebra. Example \ref{ex:FttpK} and Proposition \ref{pr:DerJ0} give:
\[
\Der(J)\subo=\Inder(J)\subo=D(x,Zx)\oplus D(w_2,Zw_3)\oplus D(w_3,Zw_1)\oplus
D(w_1,Zw_2).
\]
\end{example}

\bigskip

\subsection{Odd derivations}\label{ss:OddDerivations} Let again $J$ be the
Cheng-Kac Jordan superalgebra $J=JCK(Z,\delta)$ for a unital commutative and
associative algebra $Z$ and a derivation $\delta$ of $Z$ with $Z=Z\delta(Z)$.

\medskip

\subsubsection{}\label{sss:tildeetaa} Let $a\in Z$ and consider the odd
derivation $\eta_a$ of $K$ in Proposition \ref{pr:DerZZx}: $\eta_a(Z)=0$,
$\eta_a(x)=a$.
Then $\eta_a$ extends to an inner odd derivation $\tilde\eta_a$ of $J$ by means
of $\tilde\eta_a(x_j)=0$ for any $j=1,2,3$.

Actually, for any $f,g\in Z$, the action of the odd derivation $D(f,gx)$ is
determined by:
\begin{equation}\label{eq:DfgxJ}
D(f,gx): Z\mapsto 0,\quad x\mapsto g\delta(f),\quad x_j\mapsto 0,\ j=1,2,3.
\end{equation}
(Compare to \eqref{eq:DfgxK}.) Hence if $a=\sum g_i\delta(f_i)$
($Z=Z\delta(Z)$), we get $\tilde\eta_a=\sum D(f_i,g_ix)$.

\medskip

\subsubsection{} Hence if $\partial$ is an odd derivation in
$\Der(J)\subuno^{[\bar 0,\bar 0]}$ with $\partial(x)=a$, then
$(\partial-\tilde\eta_a)(x)=0$.

\medskip

\subsubsection{} Let then $\partial$ be an odd derivation in
$\Der(J)\subuno^{[\bar 0,\bar 0]}$ such that $\partial(x)=0$. According to
Proposition \ref{pr:DerZZx} either $\partial(Z)=0$ or the characteristic of
$\bF$ is $3$ and there is an element $a\in Z$ such that
$\partial(f)=a\delta(f)x$ for any $f\in Z$. For any $j=1,2,3$ there is an
element $b_j\in Z$ such that $\partial(x_j)=b_jw_j$ (recall
$\partial\in\Der(J)\subuno^{[\bar 0,\bar 0]}$). Then, taking indices modulo
$3$,
\[
0=\partial(x_ix_{i+1})=(b_iw_i)x_{i+1}-x_i(b_{i+1}w_{i+1})=-(b_i+b_{i+1})x_{i\times
(i+1)}.
\]
Hence $b_1=-b_2=b_3=-b_1$, so that $b_1=b_2=b_3=0$ and $\partial(x_j)=0$.

Therefore, in case $\partial(Z)=0$, then $\partial=0$, as $Z$, $x$ and the
$x_j$'s generate $J$.

Assume then that the characteristic of $\bF$ is $3$ and
$\partial(f)=a\delta(f)x$ for any $f\in Z$ for a fixed $a\in Z$. Now, for any
$f\in Z$, $\partial(fx_j)=(a\delta(f)x)x_j=-a\delta(f)w_j$ and
$0=\partial\bigl((fx_1)x_2\bigr)=\partial(fx_1)x_2=-(a\delta(f)w_1)x_2=a\delta(f)x_3$.
It follows $a\delta(Z)=0$, so $aZ=aZ\delta(Z)=0$, and $a=0$, so $\partial =0$
in this case too.

\medskip

We conclude that $\Der(J)\subuno^{[\bar 0,\bar 0]}=D(Z,Zx)=\{\tilde\eta_a: a\in
Z\}$.

\bigskip

\subsubsection{} Consider an odd derivation $\partial\in\Der(J)\subuno^{[\bar
1,\bar 0]}$, then $\partial(x_1)=a\in Z$, so $\partial(x_1)=-D(w_1,ax)(x_1)$.
As $D(w_1,ax)$ is in $\Der(J)\subuno^{[\bar 1,\bar 0]}$ too, we may substitute
$\partial$ by $\partial +D(w_1,ax)$ and assume $\partial(x_1)=0$. Then
$\partial(w_1)=\partial(x_1x)=-x_1\partial(x)\in x_1(Zw_1)=0$. For any $f\in
Z$, $\partial(f)=\partial(w_1(fw_1))=w_1\partial(fw_1)\in w_1(Zx)=0$, so
$\partial(Z)=0$. Now we have $\partial(x)=aw_1$ for some $a\in Z$, and hence
\[
\begin{split}
0=\partial\delta(f)&=\partial((fx)x)=\partial(fx)x-(fx)\partial(x)\\
 &=(faw_1)x-(fx)(aw_1)=(\delta(fa)-f\delta(a))x_1=a\delta(f)x_1.
\end{split}
\]
Therefore $a\delta(Z)=0$, so $aZ=aZ\delta(Z)=0$ and $a=0$. Thus
$\partial(x)=0$. Finally $0=\partial(x_1x_2)=-x_1\partial(x_2)$, so
$\partial(x_2)=0$ because $\{v\in Zw_3:x_1v=0\}=0$. In the same vein
$\partial(x_3)=0$, and hence $\partial=0$.

We conclude $\Der(J)\subuno^{[\bar 1,\bar 0]}=D(w_1,Zx)$, and the same
arguments give $\Der(J)\subuno^{[\bar 0,\bar 1]}=D(w_2,Zx)$ and
$\Der(J)\subuno^{[\bar 1,\bar 1]}=D(w_3,Zx)$.

\medskip

Therefore we get the next result:

\begin{proposition}\label{pr:DerJ1}
Let $J$ be the Cheng-Kac Jordan superalgebra $JCK(Z,\delta)$ for a unital
commutative and associative algebra $Z$ and a derivation $\delta$ of $Z$ with
$Z=Z\delta(Z)$. Relative to the $\bZ_2^2$-grading induced by
\eqref{eq:Z22grading} we have
\[
\begin{split}
\Der(J)\subuno^{[\bar 0,\bar 0]}&=D(Z,Zx)=\{\tilde\eta_a:a\in Z\},\\
\Der(J)\subuno^{[\bar 1,\bar 0]}&=D(w_1,Zx),\\ \Der(J)\subuno^{[\bar 0,\bar
1]}&=D(w_2,Zx),\\ \Der(J)\subuno^{[\bar 1,\bar 1]}&=D(w_3,Zx).
\end{split}
\]
\end{proposition}

\begin{corollary}\label{co:oddDer}
Let $J$ be the Cheng-Kac Jordan superalgebra $JCK(Z,\delta)$ for a unital
commutative and associative algebra $Z$ and a derivation $\delta$ of $Z$ with
$Z=Z\delta(Z)$. Then:
\begin{itemize}
\item Any odd derivation is inner: $\Der(J)\subuno=\Inder(J)\subuno$. \item
    The restriction map $\Der(J)^{[\bar 0,\bar 0]}\mapsto \Der(K):
    \partial\mapsto \partial\vert_K$ is one-to-one with image
    $\overline{\Der}(K)$. \item $\dim_\bF \Der(J)\subuno=4\dim_\bF
    Z=\dim_\bF J\subuno$.
\end{itemize}
\end{corollary}

\begin{example}\label{ex:FttpJodd}
Let $\bF$ be a field of prime characteristic $p>2$, let $Z=\bF[t:t^p=0]$ be the
algebra of truncated polynomials, and let $\delta=\frac{d\ }{dt}$ be the usual
derivative. Let $J=JCK(Z,\delta)$ be the associated Cheng-Kac Jordan
superalgebra. For any $f\in Z$ and $i=1,2,3$, $D(w_i,fx)+D(f,xw_i)+D(x,fw_i)=0$
\eqref{eq:cyclic}, and $xw_i=0$. Hence $D(w_i,Zx)=D(Zw_i,x)$ holds. Also for
any $f,g\in Z$, $D(f,gx)=\tilde\eta_{g\delta(f)}$ \eqref{eq:DfgxJ} and hence,
as in Example \ref{ex:FttpK}, $D(Z,Zx)=D(Z,x)\oplus \bF D(t,t^{p-1}x)$. It
follows from Proposition \ref{pr:DerJ1} that
$\Inder(J)\subuno=D(J\subo,x)\oplus \bF D(t,t^{p-1}x)$.
\end{example}

\medskip

\begin{remark}\label{re:DZZxj}
For any $i=1,2,3$, $D(Z,Zx_i)=0$, as for any $f,g\in Z$, $D(f,gx_i)$ acts
trivially on $Z$, $x$, and $x_j$, $j=1,2,3$.
\end{remark}

\smallskip

As a direct consequence of Corollaries \ref{co:DerJ0} and \ref{co:oddDer} we
obtain:

\begin{corollary}\label{co:dimensionInder}
Let $J$ be the Cheng-Kac Jordan superalgebra $JCK(Z,\delta)$ for a unital
commutative and associative algebra $Z$ and a derivation $\delta$ of $Z$ with
$Z=Z\delta(Z)$. Then $\dim_\bF\Inder(J)=\dim_\bF J=8\dim_\bF Z$.
\end{corollary}

\bigskip

\subsection{The Tits-Kantor-Koecher Lie superalgebra $\calK(J)$}\label{ss:TKK}
The so called Tits-Kantor-Koecher Lie superalgebra $\calK(J)$ attached to a
Jordan superalgebra $J$ is defined on the direct sum of $\Inder(J)$ and three
copies of $J$ (see \cite{JacobsonJordan}). Actually, $\calL=\calK(J)$ is a
$3$-graded Lie superalgebra $\calL=\calL_{-1}\oplus\calL_0\oplus\calL_1$, where
$\calL_{\pm 1}$ is a copy of $J$, $\calL_0$ is the Lie subalgebra of $\frgl(J)$
generated by the left multiplications by elements in $J$ and by $\Inder(J)$:
$\calL_0=L_J\oplus\Inder(J)$ (the \emph{structure Lie superalgebra} of $J$),
and the Lie bracket is given by:
\[
\begin{split}
[x_1,y_{-1}]&=L_{xy}+D(x,y),\\ [L_a,x_{\pm 1}]&=\pm(ax)_{\pm 1},\\ [d,x_{\pm
1}]&=d(x)_{\pm 1},\\ [x_1,y_1]&=[x_{-1},y_{-1}]=0,
\end{split}
\]
for any $a,x,y\in J$ and $d\in\Inder(J)$, where $x_{\pm 1}$ denotes the element
$x$ as an element in $\calL_{\pm 1}$.

\smallskip

For $J=JCK(Z,\delta)$ a Cheng-Kac Jordan superalgebra, $\calK(J)$ is the
associated Cheng-Kac Lie superalgebra $CK(Z,\delta)$.

Corollary  \ref{co:dimensionInder} immediately gives the dimension of this
Cheng-Kac Lie superalgebra:

\begin{corollary}\label{co:dimensionTKK}
Let $J$ be the Cheng-Kac Jordan superalgebra $JCK(Z,\delta)$ for a unital
commutative and associative algebra $Z$ and a derivation $\delta$ of $Z$ with
$Z=Z\delta(Z)$, and let $\calK(J)=CK(Z,\delta)$ be the associated Cheng-Kac Lie
superalgebra. Then $\dim_\bF\calK(J)=4\dim_\bF J=32\dim_\bF Z$ holds. (And
$\dim_\bF\calK(J)\subo=\dim_\bF\calK(J)\subuno=4\dim_\bF J\subo=4\dim_\bF
J\subuno=16\dim_\bF Z$.)
\end{corollary}

\bigskip


\section{A more symmetric multiplication table}

Assume from now on that our ground field satisfies $-1\in\bF^2$. (This is
always the case if the characteristic of $\bF$ is congruent to $1$ modulo
$4$.)

Consider the following elements in the Cheng-Kac Jordan superalgebra
$J=JCK(Z,\delta)$:
\[
\begin{aligned}
&&\quad v_1&=\sqrt{-1}w_1,&\quad v_2&=\sqrt{-1}w_2,&\quad v_3&=w_3,\\
y&=x,&y_1&=\sqrt{-1}x_1,& y_2&=\sqrt{-1}x_2,& y_3&=x_3,
\end{aligned}
\]
and define $y_{1\wedge 2}=-y_{2\wedge 1}=y_3$, $y_{2\wedge 3}=-y_{3\wedge
2}=y_1$, $y_{3\wedge 1}=-y_{1\wedge 3}=y_2$, $y_{i\wedge i}=0$, $i=1,2,3$. Then
$\{1,v_1,v_2,v_3\}$ is a $Z$-basis of $J\subo$, $\{y,y_1,y_2,y_3\}$ is a
$Z$-basis of $J\subuno$, $v_i^2=-1$ for any $i=1,2,3$, $v_iv_j=0$ for any $i\ne
j$, and the multiplication of even and odd elements and of odd elements is
given by the new tables:

\[
\vbox{\offinterlineskip \halign{\hfil$#$\enspace\hfil&#\vreglon
 &\hfil\enspace$#$\enspace\hfil&#\vregleta
 &\hfil\enspace$#$\enspace\hfil\cr
 J\subo J\subuno&&gy&&gy_j\cr
 \noalign{\hreglon}
 f&&(fg)y&&(fg)y_j\cr
 \noalign{\hregleta}
 fv_i&&\bigl(\delta(f)g\bigr)y_i&&(fg)y_{i\wedge j}\cr
 }}
\qquad\qquad \vbox{\offinterlineskip \halign{\hfil$#$\enspace\hfil&#\vreglon
 &\hfil\enspace$#$\enspace\hfil&#\vregleta
 &\hfil\enspace$#$\enspace\hfil\cr
 J\subuno J\subuno&&gy&&gy_j\cr
 \noalign{\hreglon}
 fy&&\delta(f)g-f\delta(g)&&-(fg)v_j\cr
 \noalign{\hregleta}
 fy_i&&(fg)v_i&&0\cr
 }}
\]
\medskip

Now the three indices $1,2,3$ play the same role. This fact, together with the
$\bZ_2^2$-grading considered above proves the next result. In it we consider
the generators $\tau_1=(1\,2)(3\,4)$, $\tau_2=(2\,3)(4\,1)$,
$\varphi=(1\,2\,3)$ and $\tau=(1\,2)$ of the symmetric group $S_4$. The choice
of these generators is due to the fact that the $\bZ_2^2$-grading above is the
decomposition into common eigenspaces for the elements in Klein's $4$-group
(generated by $\tau_1$ and $\tau_2$).

\begin{theorem}\label{th:S4action}
The symmetric group of degree $4$ embeds in the group $\Aut J$ of automorphisms
of $J$ as follows:
\begin{itemize}
\item $\tau_1=(1\,2)(3\,4)$ acts as the identity on $J^{[\bar 0,\bar
    0]}\oplus J^{[\bar 1,\bar 1]}$ and multiplying by $-1$ on $J^{[\bar
    1,\bar 0]}\oplus J^{[\bar 0,\bar 1]}$, \item $\tau_2=(2\,3)(4\,1)$ acts
    as the identity on $J^{[\bar 0,\bar 0]}\oplus J^{[\bar 1,\bar 0]}$ and
    multiplying by $-1$ on $J^{[\bar 0,\bar 1]}\oplus J^{[\bar 1,\bar 1]}$,
    \item $\varphi=(1\,2\,3)$ acts as the $Z$-linear map such that
    $\varphi(1)=1$ and
\[
\varphi(v_i)=v_{i+1},\quad \varphi(y_i)=y_{i+1},\quad \varphi(y)=y
\]
(indices modulo $3$). \item $\tau=(1\,2)$ acts as the $Z$-linear map such
that $\tau(1)=1$, $\tau(y)=y$ and
    \[
    \begin{aligned}
    \tau(v_1)&=-v_2,&\quad \tau(v_2)&=-v_1,&\quad \tau(v_3)&=-v_3\\
    \tau(y_1)&=-y_2,&\tau(y_2)&=-y_1,&\tau(y_3)&=-y_3.
    \end{aligned}
    \]
\end{itemize}
\end{theorem}

\bigskip

Note that the symmetric group $S_4$ embeds in this way in $(\Aut J)\cap
\End_Z(J)\subo$.

As a consequence, the Lie superalgebra of derivations of $J$: $\Der(J)$, and
the Lie superalgebra of inner derivations: $\Inder(J)$, are Lie superalgebras
with $S_4$-symmetry.

\smallskip

It is straightforward to check that the even derivations $\tilde\partial$ in
\eqref{eq:tildepartial} commute with the action of the symmetric group $S_4$.
Also $S_4$ acts trivially on $K=Z\oplus Zx=Z\oplus Zy$, and therefore $S_4$
acts trivially by conjugation on $\Der(J)\subuno^{[\bar 0,\bar
0]}=D(Z,Zx)=D(Z,Zy)$ (Proposition \ref{pr:DerJ1}). Since it acts trivially too
on $\Der(J)\subo^{[\bar 0,\bar 0]}=\{\tilde\partial:\partial\in\Der(K)\subo\}$
(Proposition \ref{pr:DerJ0}), the next results follows:

\begin{proposition}\label{pr:S4fixesDerJ00}
The action by conjugation of $S_4$ on $\Der(J)^{[\bar 0,\bar 0]}$ is trivial.
\end{proposition}

\bigskip


\section{Coordinate superalgebra}

Let $J$ be the Cheng-Kac Jordan superalgebra $JCK(Z,\delta)$ for a unital
commutative and associative algebra $Z$ and a derivation $\delta$ of $Z$ with
$Z=Z\delta(Z)$, over a field $\bF$ with $-1\in\bF^2$.  The Lie superalgebra of
derivations $\Der(J)$ is a Lie superalgebra with $S_4$-symmetry (the action of
$S_4$ on $\Der(J)$ given by conjugation), and therefore \cite[\S 2]{EO1} the
subspace $\Der(J)^{[\bar 1,\bar 1]}=\{ X\in\Der(J): \tau_1(X)=X=-\tau_2(X)\}$
is a superalgebra with involution (in the super sense), where the
multiplication and the involution are given by:
\begin{equation}\label{eq:mult-invo}
\begin{split}
X\cdot Y&= -\tau\bigl([\varphi(X),\varphi^2(Y)]\bigr),\\ \bar X&=-\tau(X),\\
\end{split}
\end{equation}
for any $X,Y\in \Der(J)^{[\bar 1,\bar 1]}$. Actually, this is proved in
\cite{EO1} in the algebra setting, but it extends naturally to the superalgebra
setting \cite{EO2}. The superalgebra with involution $(\Der(J)^{[\bar 1,\bar
1]},\cdot,-)$ is said to be the \emph{coordinate superalgebra} of $\Der(J)$. As
proved in \cite{EO1}, this coordinate superalgebra is structurable in case it
is unital.

The aim of this section is the proof of the following result:

\begin{theorem}\label{th:coordinate}
The superalgebra with involution $(\Der(J)^{[\bar 1,\bar 1]},\cdot,-)$ above is
isomorphic to the Jordan superalgebra $K=Z\oplus Zy$ with the trivial
involution (the identity map).
\end{theorem}
\begin{proof}
Propositions \ref{pr:DerJ0} and \ref{pr:DerJ1} give $\Der(J)\subo^{[\bar 1,\bar
1]}=D(v_1,Zv_2)$ and $\Der(J)\subuno^{[\bar 1,\bar 1]}=D(v_3,Zy)$.

For any $f\in Z$ we have
\[
\begin{split}
\tau\bigl(D(v_1,fv_2)\bigr)&=D(\tau(v_1),\tau(fv_2))=D(-v_2,-fv_1)=D(v_2,fv_1)\\
 &=-D(v_1,v_2f)=-D(v_1,fv_2);\\[6pt]
\tau\bigl(D(v_3,fy)\bigr)&=D(\tau(v_3),\tau(fy))=D(-v_3,fy)=-D(v_3,fy).
\end{split}
\]
Therefore, equation \eqref{eq:mult-invo} shows that the involution is trivial:
$\overline{D(v_1,fv_2)}=D(v_1,fv_2)$ and $\overline{D(v_3,fy)}=D(v_3,fy)$. In
particular, this shows that the multiplication in \eqref{eq:mult-invo} is
supercommutative.

Let us compute now the multiplication in $\Der(J)^{[\bar 1,\bar 1]}$ given in
\eqref{eq:mult-invo}. For any $f,g\in Z$:
\[
\begin{split}
D(v_1,fv_2)&\cdot
D(v_1,gv_2)=\bigl[D(\varphi(v_1),\varphi(fv_2)),D(\varphi^2(v_1),\varphi^2(gv_2))\bigr]\\
 &=[D(v_2,fv_3),D(v_3,gv_1)]\\
 &=D\bigl(D(v_2,fv_3)(v_3),gv_1\bigr)+D\bigl(v_3,D(v_2,fv_3)(gv_1)\bigr)\\
 &=-D(fv_2,gv_1)=D(v_1,(fg)v_2),
\end{split}
\]
\[
\begin{split}
D(v_1,fv_2)&\cdot D(v_3,gy)=
 \bigl[D(\varphi(v_1),\varphi(fv_2)),D(\varphi^2(v_3),\varphi^2(gy))\bigr]\\
 &=[D(v_2,fv_3),D(v_2,gy)]\\
 &=D\bigl(D(v_2,fv_3)(v_2),gy\bigr)+D\bigl(v_2,D(v_2,fv_3)(gy)\bigr)\\
 &=D(fv_3,gy)+D(v_2,(\delta(f)g)y_1)\\
 &=-D((gy)f,v_3)=D(v_3,(fg)y)\\
 &\qquad\quad\text{as $D(v_2,Zy_1)=0$ (trivial action on $Z$, $y$, $y_j$,
 $j=1,2,3$),}\\[6pt]
D(v_3,fy)&\cdot D(v_3,gy)
 =\bigl[D(\varphi(v_3),\varphi(fy)),D(\varphi^2(v_3),\varphi^2(gy))\bigr]\\
 &=[D(v_1,fy),D(v_2,gy)]\\
 &=D\bigl(D(v_1,fy)(v_2),gy\bigr)+D\bigl(v_2,D(v_1,fy)(gy)\bigr)\\
 &=D(v_2,(\delta(f)g-f\delta(g))v_1)=-D(v_1,(\delta(f)g-f\delta(g))v_2).
\end{split}
\]

It then follows at once that the map:
\[
\begin{split}
\Phi:K\ &\longrightarrow (\Der(J)^{[\bar 1,\bar 1]},\cdot),\\ f+gy&\mapsto
D(v_1,fv_2)+\sqrt{-1}D(v_3,gy)
\end{split}
\]
is an isomorphism of superalgebras. This concludes the proof of Theorem
\ref{th:coordinate}.
\end{proof}

\bigskip

The symmetric group $S_4$ fixes elementwise the subalgebra $\Der(J)^{[\bar
0,\bar 0]}$ (Proposition \ref{pr:S4fixesDerJ00}), and therefore $\Der(J)^{[\bar
0,\bar 0]}$ acts by derivations on the superalgebra $(\Der(J)^{[\bar 1,\bar
1]},\cdot)$, which is isomorphic to $K$. This induces a homomorphism of Lie
superalgebras:
\[
\begin{split}
\Phi^*:\Der(J)^{[\bar 0,\bar 0]}&\longrightarrow \Der(K),\\
  \partial\ &\mapsto\ \hat\partial:z\mapsto \Phi^{-1}([\partial,\Phi(z)]).
\end{split}
\]

\begin{proposition}\label{pr:DerJ00DerK}
$\Phi^*$ is an isomorphism onto $\overline{\Der}(K)$. It takes
$\Inder(J)^{[\bar 0,\bar 0]}$ onto $\Inder(K)$. Thus, we have
\[
\Der(J)^{[\bar 0,\bar 0]}\cong \overline{\Der}(K),\quad \Inder(J)^{[\bar 0,\bar
0]}\cong \Inder(K).
\]
\end{proposition}
\begin{proof}
For $\partial\in\overline{\Der}(K)\subo$, let $\tilde\partial$ be the
corresponding derivation of $J$ as in \eqref{eq:tildepartial}, so that
\[
\tilde\partial\vert_{K}=\partial,\quad\tilde\partial(fv_i)=\mu(f)v_i,\quad
\tilde\partial(fy_i)=(\mu(f)-af)y_i,
\]
for any $f\in Z$ and $i=1,2,3$, where $\mu=\partial\vert_Z$ and $a\in Z$ is the
element such that $[\mu,\delta]=2a\delta$, and which satisfies
$\partial(y)=ay$.

Then for any $f\in Z$:
\[
[\tilde\partial,D(v_1,fv_2)]=D(\tilde\partial(v_1),fv_2)+D(v_1,\tilde\partial(fv_2))
=D(v_1,\partial(f)v_2),
\]
because $\tilde\partial(v_i)=0$ for $i=1,2,3$. Hence
$[\tilde\partial,\Phi(f)]=\Phi(\partial(f))$. Also,
\[
[\tilde\partial,D(v_3,fy)]=D(v_3,\tilde\partial(fy))=D(v_3,\partial(fy)),
\]
so $[\tilde\partial,\Phi(fy)]=\Phi(\partial(fy))$. We conclude
$\partial=\Phi^*(\tilde\partial)$.

Now for any $\partial\in\overline{\Der}(K)\subuno$, there is an element $a\in
Z$ such that $\partial=\eta_a$ (Proposition \ref{pr:DerZZx}, recall
$\eta_a(Z)=0$). Thus $\partial$ extends to the odd derivation $\tilde\eta_a$ in
Subsection \ref{sss:tildeetaa} with $\tilde\eta_a(y_i)=0$ for any $i=1,2,3$.
Then $\tilde\eta_a(v_i)=\tilde\eta_a(y_iy)=-y_i\tilde\eta_a(y)=-ay_i$ for
$i=1,2,3$, and for any $f\in Z$:
\[
[\tilde\eta_a,D(v_1,fv_2)]=D(-ay_1,fv_2)+D(v_1,-afy_2)=0
\]
because $D(v_i,Zy_j)=D(Z,Zy_j)=0$ for $i\ne j$ (they act trivially on $Z$, $y$,
$y_k$, $k=1,2,3$). Also,
\[
\begin{split}
[\tilde\eta_a,D(v_3,fy)]&=D(-ay_3,fy),\quad\text{(as $D(v_3,Z)=0$)}\\
  &=-D(v_1(ay_2),fy)\\
  &=D((ay_2)(fy),v_1)-D((fy)v_1,ay_2)\\
  &=D((af)v_2,v_1)\qquad \text{($(Zy)v_1=0$)}\\
  &=-D(v_1,(af)v_2)=-\Phi(af)\\
  &=-\Phi(\eta_a(fy)).
\end{split}
\]
Therefore, since $\Phi(fy)=\sqrt{-1}D(v_3,fy)$, we conclude
$\sqrt{-1}[\tilde\eta_a,\Phi(z)]=\Phi(\eta_a(z))$ for any $z\in K$, or
$\eta_a=\Phi^*(\sqrt{-1}\tilde\eta_a)$. This completes the proof.
\end{proof}

\bigskip


\section{Derivations and the Tits-Kantor-Koecher construction}\label{se:TKK}

Given a Jordan superalgebra $J$ and a subalgebra $\frd$ of $\Der(J)$ such that
$\Inder(J)\leq \frd\leq\Der(J)$, consider the orthogonal Lie algebra $\frso_3$
of skewsymmetric $3\times 3$ matrices over $\bF$, and the vector space
\[
\frg=\bigl( \frso_3\otimes J\bigr)\oplus \frd,
\]
with superanticommutative bracket given by
\begin{itemize}
\item $\frd$ is a subalgebra of $\frg$, \item $[\partial, A\otimes
    x]=A\otimes \partial(x)$, for any $\partial\in\frd$, $A\in\frso_3$ and
    $x\in J$, \item $[A\otimes x,B\otimes y]=[A,B]\otimes xy
    +\frac{1}{2}\trace(AB)D(x,y)$ for any $A,B\in\frso_3$ and $x,y\in J$.
\end{itemize}
Then \cite{Tits62} $\frg$ is a Lie superalgebra. (Actually the work
\cite{Tits62} deals only with the non-super case, but all the arguments there
are valid in the super setting \cite{EO2}.)

This Lie superalgebra $\frg$ thus obtained will be denoted by $\calT(J,\frd)$.
For $\frd=\Inder(J)$ we will write simply $\calT(J)$.

If $\bF$ is a field of prime characteristic $p$, then $\frso_3$ is split, that
is, isomorphic to $\frsl_2$. (One can check this directly, or using that
$\frso_3$ consists of the trace zero elements of the classical quaternion
algebra $\bH=\bF 1\oplus\bF i\oplus\bF j\oplus \bF k$, with $i^2=j^2=-1$, and
$ij=-ji=k$, but the norm of this algebra represents $0$ already over the prime
subfield $\bF_p=\bZ/p\bZ$, and hence it is isomorphic to $\Mat_2(\bF)$.) If
$\bF$ is a field of characteristic $0$ in which $\bH$ is split (for instance,
if $-1\in\bF^2$), then also $\frso_3$ is isomorphic to $\frsl_2$. In this case
take a standard basis $\{e,f,h\}$ of $\frsl_2$ with $[h,e]=2e$, $[h,f]=-2f$ and
$[e,f]=h$. The assignment:
\[
e\otimes x\leftrightarrow x_1,\quad \frac{1}{2}h\otimes x\leftrightarrow
L_x,\quad \frac{1}{2}f\otimes x\leftrightarrow x_{-1},\quad d\leftrightarrow
d,
\]
for $x\in J$ and $d\in \Inder(J)$, provides an isomorphism from $\calT(J)$ onto
$\calK(J)$ (see \ref{ss:TKK}). Therefore, $\calT(J)$ is nothing else but the
Tits-Kantor-Koecher superalgebra of the Jordan superalgebra $J$.

\smallskip

Note that $\frso_3$ has a basis $\{E_1,E_2,E_3\}$ with $[E_i,E_{i+1}]=E_{i+2}$
(indices modulo $3$).

\smallskip

Our last result relates the Lie superalgebra of derivations of the Cheng-Kac
Jordan superalgebra $JCK(Z,\delta)$ to the Tits-Kantor-Koecher superalgebra
attached to the Jordan superalgebra $K=Z\oplus Zx$:

\begin{theorem}\label{th:DerJTKK}
Let $J$ be the Cheng-Kac Jordan superalgebra $JCK(Z,\delta)$ for a unital
commutative and associative algebra $Z$ and a derivation $\delta$ of $Z$ with
$Z=Z\delta(Z)$, over a field $\bF$ with $-1\in\bF^2$.

Then the Lie superalgebra of derivations $\Der(J)$ is isomorphic to the
Tits-Kantor-Koecher superalgebra $\calT(J,\overline{\Der}(K))$, while the Lie
superalgebra of inner derivations $\Inder(J)$ is isomorphic to $\calT(K)$.
\end{theorem}
\begin{proof}
For any $z\in K$, let $\iota_3(z)=\Phi(z)\in\Der(J)^{[\bar 1,\bar 1]}$ and
$\iota_1(z)=\varphi(\iota_3(z))\in\Der(J)^{[\bar 1,\bar 0]}$ and
$\iota_2(z)=\varphi^2(\iota_3(z))\in\Der(J)^{[\bar 0,\bar 1]}$. Then the fact
that $\Phi$ is an isomorphism is equivalent to the condition
\[
\iota_3(z_1z_2)=[\iota_1(z_1),\iota_2(z_2)],
\]
for any $z_1,z_2\in K$. The action of the automorphism $\varphi$ gives then
\[
\iota_i(z_1z_2)=[\iota_{i+1}(z_1),\iota_{i+2}(z_2)]
\]
for any $i=1,2,3$ (indices modulo $3$) and any $z_1,z_2\in K$.

Following some of the ideas in \cite[\S 7]{EO2}, identify the element
$\iota_i(z)$ with $E_i\otimes z\in \frso_3\otimes K$ )($i=1,2,3$), and identify
$\Der(J)^{[\bar 0,\bar 0]}$ with $\frd=\overline{\Der}(K)$ through the
isomorphisms $\Phi^*$, to get an identification $\Der(J)\simeq (\frso_3\otimes
K)\oplus \frd$.

The Jacobi identity gives
\[
\begin{split}
&[[\iota_i(z),\iota_i(z')],\iota_{i+1}(z'')]\\ &\qquad
=[\iota_i(z),\iota_{i+2}(z'z'')]+(-1)^{\bar z'\bar
z''}[\iota_{i+2}(zz''),\iota_i(z)]\\ &\qquad =-(-1)^{\bar z(\bar z'+\bar
z'')}\iota_{i+1}((z'z'')z)+(-1)^{\bar z'\bar z''}\iota_{i+1}((zz'')z')\\
&\qquad =-\iota_{i+1}(D(z,z')(z'')).
\end{split}
\]
But $\trace(E_i^2)=-2$ for any $i=1,2,3$, and this shows that the
identification $\Der(J)\simeq (\frso_3\otimes K)\oplus \frd$ above is an
isomorphism of Lie superalgebras $\Der(J)\cong\calT(K,\overline{\Der}(K))$,
which restricts to an isomorphism $\Inder(J)\simeq \calT(J)$, because
$\Inder(J)^{[\bar 0,\bar 0]}=D(y,Zy)\oplus D(Z,Zy)$ corresponds through
$\Phi^*$ to $\Inder(K)$.
\end{proof}

%

\end{document}